\documentclass[11pt,a4paper]{article}
\usepackage{amsmath,amssymb,amsthm}
\usepackage[margin=30mm]{geometry}
\usepackage[hidelinks]{hyperref}
\usepackage{tikz}
\newtheorem{theorem}{Theorem}[section]
\newtheorem{lemma}[theorem]{Lemma}
\newtheorem{corollary}[theorem]{Corollary}
\newtheorem{proposition}[theorem]{Proposition}
\theoremstyle{definition}

\theoremstyle{remark}
\newtheorem{remark}[theorem]{Remark}
\newtheorem{question}[theorem]{Question}
\newcommand{\Z}{\mathbb Z}
\newcommand{\N}{\mathbb N}
\newcommand{\CP}{\mathrm{CP}}
\newcommand{\id}{\mathrm{id}}
\newcommand{\Homeo}{\mathrm{Homeo}}
\newcommand{\Aut}{\mathrm{Aut}}
\newcommand{\RTM}{\mathrm{RTM}}

\title{Conjugacy problems in higher-dimensional Thompson groups}
\author{Ville Salo\\\texttt{vosalo@utu.fi}}
\date{}
\begin{document}
\maketitle
\begin{abstract}
We prove that topological conjugacy of elements of the Brin--Thompson group
$nV$ is undecidable for every $n\geq2$. From a Turing machine $T$, the reduction produces two elements of $nV$, which are conjugate by an involution in $nV$ if $T$ halts on the empty tape. If $T$ does not halt on the empty tape, the topological dynamics of the two group elements are distinguished by their visits to minimal subsystems of the set of (germ-)aperiodic points. In these groups, and any finitely-generated supergroups inside the homeomorphism group of Cantor space, we also obtain undecidability of the Whitehead problem (the problem of being in the same automorphism orbit), and undecidability of group-theoretic conjugacy.
\end{abstract}

\section{Introduction}

The conjugacy problem asks whether two elements of a group are conjugate by another element of that group. For a group acting on a Cantor space, one may also ask whether its elements are conjugate by an arbitrary homeomorphism of the space. An obstruction to conjugacy within the group need not obstruct topological conjugacy.

The conjugacy problem for the Brin--Thompson group $2V$ and for groups of reversible one-head machines was proved undecidable in \cite{Salo}. Here we strengthen this conclusion to topological conjugacy. The nonconjugacy survives taking a product with an identity map on any nonempty compact metrizable space, thus we obtain for example undecidability of conjugacy and topological conjugacy also in $nV$, by using the standard embedding $2V\leq nV$.

The construction is based on the reversible simulation of Kari and Ollinger \cite{KaOl08},  and is in fact almost equal to the construction used in \cite{Salo} to prove undecidability of conjugacy, although instead of conveyor belts we use simply simulated tapes.

\subsection{Conjugacy relations and the main result}

Let $G\leq\Homeo(Z)$, where $Z$ is a Cantor space. Define
\begin{align*}
\CP_1(G)&=\{(f,g)\in G^2:\exists h\in G,\ hfh^{-1}=g\},\\
\CP_2(G)&=\{(f,g)\in G^2:\exists\alpha\in\Aut(G),\ \alpha(f)=g\},\\
\CP_3(G)&=\{(f,g)\in G^2:\exists h\in\Homeo(Z),\ hfh^{-1}=g\}.
\end{align*}
These are the conjugacy, automorphism-conjugacy (a.k.a.\ the Whitehead problem) and topological conjugacy, respectively. We also define a more technical variant $\CP_4(G)$: $(f,g)\in\CP_4(G)$ if there exist a nonempty compact metrizable space $Y$ and a homeomorphism $h$ of $Z \times Y$ such that
\[
 h \circ (f\times\id_Y) \circ h^{-1}=g\times\id_Y.
\]
We emphasize that $(f, g) \not\in \CP_4$ means topological non-conjugacy of $f \times \id_Y$ and $g \times \id_Y$ for \emph{every} choice of $Y$.

We have $\CP_1(G) \subset \CP_2(G)$ and $\CP_3(G) \subset \CP_4(G)$ for any $G$, and as explained below, Rubin's reconstruction theorem implies $\CP_2(G) \subset \CP_3(G)$ for a large class of groups $G$ (in particular for the groups $nV$, $n \geq 1$).

Let $X\subseteq\Sigma^{\Z}$ be an infinite mixing shift of finite type. The group $G_X$ consists of homeomorphisms given by finite local rewriting around the origin, allowing changes in length on either side of the origin. Its subgroup $G_X^u$ consists of the length-preserving local rewrites; precise definitions appear in Section~\ref{sec:prelim}. For the binary full shift, $G_X$ is $2V$ and $G_X^u$ is the group of one-state reversible one-head machines from \cite{BaKaSa16,BaKaSa23}.

\begin{theorem}\label{thm:MainIntro}\label{thm:Main}
Fix an infinite mixing SFT $X$. Given a Turing machine $T$, one can effectively construct $f,g\in G_X^u$ such that:
\begin{enumerate}
\item if $T$ halts on empty input, an involution in $G_X^u$ conjugates
$f$ to $g$;
\item if $T$ does not halt on empty input, then for every nonempty
compact metrizable space $Y$, the homeomorphisms $f\times\id_Y$ and
$g\times\id_Y$ are not topologically conjugate.
\end{enumerate}
\end{theorem}

\begin{corollary}\label{cor:intro}
For each $n\geq2$, the relations $\CP_i(nV)$, $1\leq i\leq4$, are undecidable in the standard action. The relations $\CP_1(nV)$ and $\CP_2(nV)$ are $\Sigma^0_1$-complete. The same four undecidability conclusions hold for $\RTM(k,s)$, $k\geq2$, $s\geq1$, acting in the moving-tape model, where $k$ is the tape alphabet size and $s$ is the number of states.
\end{corollary}

The undecidability also holds for supergroups in these actions whenever finite descriptions of the constructed elements can effectively be translated into inputs for the supergroup. We spell out this effectiveness condition and the local density needed for automorphism-conjugacy in Section~\ref{sec:consequences}.

We base the construction directly on \cite{KaOl08}. We note that there are other more modern techniques that could likely be used to achieve similar feats (and perhaps allow stronger control on the dynamical properties of the homeomorphisms $f,g$ we construct). See e.g.\ \cite{SMART,TorresMixing,GajardoOllingerTorres}.

\section{Preliminaries}\label{sec:prelim}

\subsection{Local rewrites and one-head machines}

A subshift $X\subseteq\Sigma^{\Z}$ is a closed shift-invariant subset of the full shift $\Sigma^\Z$ over a finite alphabet $\Sigma$. It is of finite type if membership is determined by avoiding finitely many finite words. A memory radius $r$ means that allowed words of length $r+1$ determine membership. Mixing for a subshift $X$ means that any two allowed words can be joined by an allowed word of every sufficiently large length, in the sense that (writing $u \sqsubset X \iff \exists x \in X: x_{[0, |u|-1]} = u$) we have
\[ \forall u, v \sqsubset X: \exists n_0: \forall n \geq n_0: \exists w\in\Sigma^n: uwv \sqsubset X. \]
For a word $u\in\Sigma^*$ and $i\in\Z$, write
\[
 [u]_i=\{x\in X:x_{[i,i+|u|-1]}=u\}.
\]
These cylinders form a clopen basis for $X$. An infinite mixing SFT has no isolated points and has positive topological entropy.

Write $xu.vy$ for a configuration $z$ with $u$ immediately to the left of the origin, $z_{[-|u|,-1]} = u$, and $v$ immediately to its right, $z_{[0,|v|-1]} = v$. Define $G_X$ as the group of homeomorphisms $f$ such that for some $N$, there is a finite description
\[
 xu.vy\longmapsto xF_1(uv).F_2(uv)y,\qquad |u|=|v|=N,
\]
where the output words have bounded length and the resulting map is a
homeomorphism of $X$. It is easy to check that $G_X$ is a subgroup of $\Homeo(X)$ under function composition.

The subgroup $G_X^u$ of $G_X$ consists of those maps for which the description can be chosen with $|F_1(uv)|+|F_2(uv)|=2N$ for every allowed input word. Such a map changes only boundedly many tape cells and then shifts the tape by a bounded amount depending on the local configuration.

For a finite nonempty state set $S$, we also consider stateful versions $G_{X,S}$ and $G^u_{X,S}$ of the above groups. These groups act on $X \times S$ and can also change the state by the local rule. Precisely, elements of $G_{X, S}$ are now homeomorphisms of the form
\[
 (xu.vy, s) \longmapsto (xF_1(uvs).F_2(uvs)y, F_3(uvs)),\qquad |u|=|v|=N,
\]
and $G^u_{X, S}$ is defined by the same condition as above.

A \emph{one-head machine} is an element of $G^u_{X, S}$, thought of as a dynamical system acting on $X \times S$. (More properly, this should be called a reversible one-head machine, but we do not consider non-reversible machines of any kind in this paper.) A \emph{partial one-head machine} is a homeomorphism between clopen subsets of $X\times S$ given by a length-preserving local rule as above. Write $\mathcal G_{X,S}$ for the groupoid whose objects are clopen subsets of $X\times S$ and whose arrows are local homeomorphisms between them, and $\mathcal G^u_{X,S}$ for its length-preserving subgroupoid. Groupoid composition is defined when the target of one arrow equals the source of the next; otherwise we may first restrict to the appropriate clopen subsets. The automorphism groups of the full object $X\times S$ are $G_{X,S}$ and $G^u_{X,S}$. We omit $S$ when it is a singleton. In much of the literature, elements of $G^u_{X,S}$ are simply called \emph{Turing machines}.

In the study of one-head machine dynamics, one also considers another type of dynamics, namely the \emph{moving-head dynamics}, where a one-head machine $f \in G^u_{X, S}$ instead acts on $X \times S \times \Z$, see \cite{KurkaTuringMachineDynamics} for the definition for classical Turing machines, and \cite{BaKaSa16,BaKaSa23} for general one-head machines. This model is not available for $G_{X, S}$, and we do not discuss decidability topological conjugacy of the action of $G^u_{X, S}$ in this model, but this model is helpful in proofs, as it allows us to talk about various positions on the tape visited by the head (of course, one can also track visited positions also in the case of $X \times S$, but the definition is less obvious).

For $X=\{0,1\}^{\Z}$, read the two half-tapes outward from the origin.
They identify $X$ with the square of the one-sided Cantor space. Local
rewrites become finite prefix replacements on clopen rectangles, giving
$G_X=2V$ \cite{Brin}. In $nV$, the same transformations act on the first two
coordinates and fix the remaining $n-2$. All these conversions of finite
descriptions are effective. Equality of local maps on an SFT is decidable:
refine to a common finite window and compare their local rules. In
particular the word problem of $nV$ is decidable from its prefix tables.

\subsection{Rubin's reconstruction theorem}
For $G\leq\Homeo(Z)$ and an open set $U\subseteq Z$, put
\[
G_U=\{h\in G:h|_{Z\setminus U}=\id\}.
\]
We call the action \emph{locally dense} if for every $x \in U$, the closure
of $G_Ux$ contains a neighborhood of $x$. In any group, we use the notation ${}^h f = h f h^{-1}$ for conjugation.

\begin{theorem}[Rubin \cite{Rubin}]
Let $G\leq\Homeo(Z)$ and $G'\leq\Homeo(Z')$ act locally densely on
locally compact Hausdorff spaces without isolated points. Then every
isomorphism $\alpha:G\to G'$ comes from a homeomorphism $h:Z\to Z'$ satisfying $\alpha(f) = {}^h f$ for every $f\in G$.
\end{theorem}

\begin{lemma}\label{lem:Inclusions}
The action of $G_X^u$ on an infinite mixing SFT $X$ is locally dense.
\end{lemma}

\begin{proof}
We show local density of the action of $G_X^u$. Let $X \subset A^\Z$. Fix $x\in U$ and a cylinder $[u]_i, u \in A^*, i \in \Z$, with $x \in [u]_i \subseteq U$. We show that the $G_U$-orbit of $x$ is dense in $[u]_i$. Given a nonempty subcylinder $W\subseteq [u]_i$, mixing allows us to replace a sufficiently large finite central word of $x$ by a word of the same length that realizes $W$ and has the same first and last $r$ symbols as the original word, where $r$ is a memory radius. The two words agree on the coordinates defining $[u]_i$. Their equal end contexts make them exchangeable in any exterior context. Swapping them at this fixed interval, and fixing everything else, defines a local involution in $G_X^u$ supported in $[u]_i$. If the words coincide, the identity already suffices. Thus the orbit meets $W$.
\end{proof}

\begin{lemma}
Let $G \leq \Homeo(X \times S)$ be any group containing $G_{X, S}^u$, or let $G\leq \Homeo((\{0,1\}^\N)^n)$ contain $nV$ for $n \geq 1$. Then $G$ satisfies
\[
\CP_1(G)\subseteq\CP_2(G)\subseteq\CP_3(G)\subseteq\CP_4(G).
\]
\end{lemma}

\begin{proof}
The natural actions of $G_{X, S}^u$ and $nV$ are locally dense. Local density passes to supergroups, so the action of $G \supset G^u_X$ is also locally dense. We show $\CP_2(G)\subseteq\CP_3(G)$, as the other inclusions follow directly from the definitions. Suppose $(f, g) \in \CP_2(G)$, i.e.\ there is an automorphism $\alpha \in \Aut(G)$ such that $\alpha(f) = g$. By Rubin's theorem, the isomorphism $\alpha:G\to G$ comes from a homeomorphism $h : X \to X$ satisfying $\alpha(v) = {}^h v$ for every $v\in G$, and in particular $g={}^h f$. This of course means exactly that $f$ and $g$ are topologically conjugate, i.e.\ $(f, g) \in \CP_3(G)$.
\end{proof}

\section{Overview}

We begin with an overview of the construction. The construction is based on the following ``technology stack'':
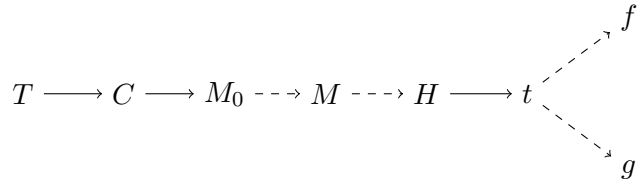
\begin{center}
\begin{tikzpicture}
\node (T) at (0,0) {$T$};
\node (C) at (1.33,0) {$C$};
\node (M0) at (2.66,0) {$M_0$};
\node (M) at (4,0) {$M$};
\node (H) at (5.33,0) {$H$};
\node (t) at (6.66,0) {$t$};
\node (f) at (8,1) {$f$};
\node (g) at (8,-1) {$g$};
\draw (T) edge[->] (C);
\draw (C) edge[->] (M0);
\draw (M0) edge[dashed, ->] (M);
\draw (M) edge[dashed, ->] (H);
\draw (H) edge[->] (t);
\draw (t) edge[dashed, ->] (f);
\draw (t) edge[dashed, ->] (g);
\end{tikzpicture}
\end{center}
By solid arrows in this diagram, we indicate ``simulation in software'' i.e.\ simulation of the previous machine on some kind of encoded configurations, while dashed arrows denote other types of modifications.

In the diagram (and the discussion below), $T$ is any Turing machine, $C$ is a reversible counter machine simulating $T$, $M_0$ is a partial one-head machine simulating $C$, $M$ is another partial one-head machine simulating $C$, modified from $M_0$ by Kari-Ollinger's reversible variant of Hooper's trick \cite{KaOl08} (i.e.\ adding recursive calls). The one-head machine $H$ is a completion of $M$ using a global time-arrow. The homeomorphism $t$ is a simulation of $H$ in an arbitrary $G^u_X$. Finally, $f, g$ are minor modifications of $t$ which have slightly different behavior in how they toggle certain bits encoded on the configurations, enough to make them dynamically distinct if the original machine does not halt.

The transition from $T$ to $C$ is due to Morita \cite{Morita}, and we do not recall it. The substack between $C$ and $H$ is as in \cite{KaOl08}, but we recall it in some detail in Section~\ref{sec:machines}, as we need to prove some basic dynamical properties for $H$ which were not explicitly stated in \cite{KaOl08}.

The construction of the final homeomorphisms $f,g \in G^u_X$ is almost the one from \cite{Salo}, although we use tape segments instead of conveyor belts. This is explained in Section~\ref{sec:encoding}.

\section{The one-head machines used in the reduction}\label{sec:machines}

We distinguish the input machine, the counter machine, the two one-head
machines, and the final homeomorphisms. Let $T$ be the input Turing machine.
Apply the standard reversible simulation used by Kari and Ollinger
\cite{KaOl08} (based on \cite{Morita}) to obtain a reversible two-counter machine $C$, with its distinguished initial configuration and mortal inverse, whose canonical computation from empty counters is finite exactly when $T$ halts on empty input. The inverse of $C$ is \emph{mortal}: it has no infinite backward computation, from any counter configuration. This additional property is supplied by the clock-counter modification in \cite[Remark after Theorem 1]{KaOl08}. The machine $T$ plays no
further role below.

Let $A$ be the tape alphabet and $S$ the state set used for the Kari--Ollinger simulation of $C$, and put $X'=A^{\mathbb Z}$. We may include a hard error symbol $\#$: it is never rewritten, and any partial computation attempting to read it is undefined. We write $M_0$ for the \emph{partial one-head machine} (a homeomorphism defined between two clopen subsets of the Cantor space $X' \times S$) that simulates $C$ on tape segments
of the schematic form
\[
       @\,0^m\,1\,0^n\,2.
\]
The head of $M_0$ carries the current state of $C$. The two blocks of zeroes
encode the two counter values, so the head sweeps between the delimiters to
test and change them. In particular, $@12$ denotes a fresh simulation with both
counters zero.

There is also an \emph{internal} direction $d\in\{+,-\}$ of computation in each state of $M_0$. When $d=+$, $M_0$ simulates $C$ as described above, while if $d=-$, $M_0$ simulates the inverse $C^{-1}$. An unexpected symbol just outside the edge of a simulated counter zone, whose size we attempt to increase, reverses the internal direction and thus retraces the current search. We say that in this case, the simulation \emph{runs out of space}. The one-sided reversal here can be compared with Figure~\ref{fig:completion}, where both endpoints are reflected. The machine does not change $d=-$ to $d=+$ in any situations (note that we are defining a partial homeomorphism!).

If a hard error symbol $\#$ is encountered outside the simulated counter zone, or an incorrect state \emph{within} a counter simulation is encountered (e.g.\ the machine is looking for the second counter, but finds instead the symbol $1$), or our one-head machine simulates a halting transition of $C$, then $M_0$ instead terminates the relevant partial computation, i.e.\ its image is undefined.

We now define the (still partial) one-head machine $M$, obtained from $M_0$ by inserting recursive simulations of $C$ from the initial configuration, whose purpose is to eliminate long delimiter searches in the simulation of $C$. Its state set is the same as that of $M_0$. When a search enters a sufficiently long zero block, $M$ writes a tagged $@_s12$, where the tag $s$ records the suspended state of the parent computation, and starts a fresh simulation of $C$. One can concretely achieve this by simply composing $M_0$ with a homeomorphism that inspects the current cell and two cells to the right, and if all of them contain $0$-symbols it inserts $@_s12$ and changes the current state to simulate the initial state of the canonical computation of the counter machine.

We note that the internal direction $d$ of time is included in the suspended state $s$; thus even if the internal direction in the state $s$ is backwards, the fresh simulation of $C$ runs forward, with its own internal direction initialized to $+$. The parent's value of $d$ remains stored in $s$ and is restored on return.

When that child search runs out of space, it retraces its steps, erases the tagged word, and resumes the parent search by executing its suspended step, so that the same call is not immediately reinserted. The insertion and erasure can be expressed by a local involution exchanging equal-length clopen patterns, and the new partial homeomorphism belongs to $\mathcal G^u_{X',S}$.


Recall that the \emph{terra cognita} \cite[Appendix A]{KaOl08} is a context-free language describing the ``expected'' configurations our machines will see. Thus it contains words $@0^m10^n2$ with an additional marker for the head position, and is closed under inserting such words in the $0$-areas of another enclosing word of the form $@0^{m'}10^{n'}2$. The area that our machine has inspected must always match terra cognita, until the machine halts.

There is a subtlely here that was not present in \cite{KaOl08}: as we define $M$ as a homeomorphism, it is looking at coordinates it has not actually visited, i.e.\ not every observed or changed symbol has ever been at the origin. It is important that if the machine is inside a consistent $@0^m10^n2$ area, it does not halt even if it in this sense ``sees'' an error symbol just outside the zone delimiters, unless it actually tries to extend the zone onto this symbol. (This convention matters when we later place a segment boundary just to the left of such a delimiter.)

In the proof of Lemma~\ref{lem:one-visit}, we will be interested in the situations where the machine, under forward $M_0$-dynamics, enters a left delimiter $@$ at position $0$ in the moving head dynamics, and immediately turns back. Pick a state $\gamma$ where this happens. We may assume that the head does not spend multiple steps at the $@$-symbol, so both the next step and the inverse of the preceding step must move the head into the first zero block, and neither may inspect a cell to the left of the delimiter. Thus there is one actual visit to the delimiter. Note that in the $M$-dynamics, both before and after the visit, the machine will (assuming the counters are large enough), immediately insert a fresh simulation $@_t12$, but the entire simulation of that orbit will stay to the right of coordinate $0$.

We reserve $M$ for the \emph{partial} machine described here, including the recursive calls. The following step makes $M$ complete. It is essentially equal to the internal time arrow construction already used above (and also already appears in the proof of Theorem 8 in \cite{KaOl08}), but since this new bit will be visible in future steps and plays a crucial role in the proof, we give the definition precisely.

Regard $M$ as a homeomorphism $M:D\longrightarrow R$ between clopen subsets of
$X' \times S$. Introduce a new bit $a\in\{0,1\}$ in the head state replacing $S$ by $S \times \{0,1\}$, and define the \emph{completed one-head machine} $H$ by
\[
\begin{aligned}
 H(y,0)&=\begin{cases}(M(y),0),&y\in D,\\(y,1),&y\notin D,\end{cases}\\
 H(y,1)&=\begin{cases}(M^{-1}(y),1),&y\in R,\\(y,0),&y\notin R.\end{cases}
\end{aligned}
\]
The image with bit $0$ is partitioned according to $R$, and the image with bit $1$ according to $D$, so $H$ is a homeomorphism in $G^u_{X',S \times\{0,1\}}$. We call $a$ the \emph{global arrow (of time)}.

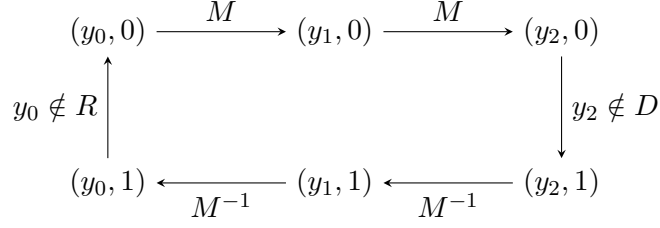
\begin{figure}[ht]
\centering
\begin{tikzpicture}[>=stealth]
\node (a0) at (0,1) {$(y_0,0)$};
\node (a1) at (3,1) {$(y_1,0)$};
\node (a2) at (6,1) {$(y_2,0)$};
\node (b0) at (0,-1) {$(y_0,1)$};
\node (b1) at (3,-1) {$(y_1,1)$};
\node (b2) at (6,-1) {$(y_2,1)$};
\draw[->] (a0)--node[above] {$M$}(a1);
\draw[->] (a1)--node[above] {$M$}(a2);
\draw[->] (a2)--node[right] {$y_2\notin D$}(b2);
\draw[->] (b2)--node[below] {$M^{-1}$}(b1);
\draw[->] (b1)--node[below] {$M^{-1}$}(b0);
\draw[->] (b0)--node[left] {$y_0\notin R$}(a0);
\end{tikzpicture}
\caption{Completion of a finite partial orbit $y_0\mapsto y_1\mapsto y_2$. At an endpoint the global arrow is flipped. For an infinite end, the endpoint reflection would be absent.}
\label{fig:completion}
\end{figure}
The involution $\rho(y,a)=(y,1-a)$ satisfies $\rho H\rho=H^{-1}$.

For a self-homeomorphism $h$ of a Cantor space, a point $x$ is \emph{germ-aperiodic} if, for every $p>0$, the germ of $h^p$ at $x$ is nontrivial. We always denote the set of germ-aperiodic points by $\mathcal{A}(h)$. Equivalently, no positive power of $h$ restricts to the identity on a neighborhood of $x$. This property is preserved by arbitrary topological conjugacies. The set of $\mathcal{A}(h)$ is invariant and closed (by definition its complement is the set of points $x$ such that for some open set $U \ni x$ we have $h^n|_U = \id|_U$ for some $n \geq 1$).

\begin{remark}
For a one-head machine on a full shift, germ-aperiodicity in the moving-tape
model of \cite{KurkaTuringMachineDynamics} corresponds to aperiodicity in the moving-head model. Indeed, a moving-head period visits only finitely many cells and restores their contents, the head state, and the head position. The same finite computation
therefore gives the identity on a cylinder in the moving-tape model. Conversely, a moving-tape period may have nonzero net head displacement: on that particular spatially periodic tape, shifting the tape back to the head hides the displacement. Such a point is \emph{germ-aperiodic}, since nearby tapes with different remote tails detect every nonzero displacement. The same local argument applies to our block encodings in an infinite mixing SFT.
\end{remark}

\begin{lemma}\label{lem:KO-two-arrows}
The one-head machine $H$ can be constructed effectively from $T$ with the following properties.
\begin{enumerate}
\item If $T$ halts, $H$ has finite order in $G^u_{X',S\times\{0,1\}}$.
\item The global arrow changes at most once on every germ-aperiodic $H$-orbit.
\item If $T$ does not halt, there are two minimal subsystems $K_0,K_1\subset\mathcal{A}(H)$, with $K_a$ contained in the clopen global-arrow layer $a$. Every infinite time tail of a germ-aperiodic $H$-orbit on which the global arrow has value $a$ accumulates on $K_a$. These are the only minimal subsystems contained in $\mathcal A(H)$.
\end{enumerate}
\end{lemma}

\begin{proof}
The first assertion is the periodicity half of the Kari--Ollinger
construction: when the initial computation of $C$ terminates, any recursive call with sufficient space actually leads to halting of $M$, and two such terminations lead to periodicity of $H$. In on the other hand all recursive calls have bounded space for their computation, then there must be a uniform bound on all searches, and thus $H$ is periodic,

For the second assertion, the global arrow changes only at an endpoint of a simulation of a partial orbit of $M$. Two changes join the forward and backward traversals of such a partial orbit. The resulting orbit for $H$ would then be finite, and not even germ-aperiodic, since we clearly have an open neighborhood where a power of $H$ is trivial.

We now prove the third assertion. A \emph{fresh-call configuration} has tape
\[
             \cdots 000@_s12000\cdots
\]
and the head at the designated entry position in the fixed initial child
state. The symbol $s$ is the suspended state of the parent;
it includes its counter-machine state, search phase, and internal direction $d$ (note that it does not contain the global arrow $a$, which is instead always carried in the second component of the state of $H$).

We first explain why fresh calls occur with arbitrarily large clean (i.e.\ zero) neighborhoods in the simulation of the canonical computation of $C$. First, consider the corresponding simulation of $C$ under the action of $M_0$. Since $T$ does not halt, $C$ has mortal inverse, and thus at least one counter must grow arbitrarily large.

In the corresponding simulation of $M$, we follow the steps of $M_0$ until possibly a recursive call is started by inserting $@_s12$ for some $s$. But this begins another simulation of $C$, where counters eventually grow, thus the simulation eventually runs out of space. Since the canonical computation (nor canonical computations started in its further recursive calls) does not halt, the action of $M$ does not flip the global arrow $a$, and eventually the recursive call must return to its starting point, $@_s12$ is erased from the tape, and the simulation of $M_0$ takes another step.

Thus, eventually in this ``top'' simulation of $C$ following the $M_0$-orbit, counters also get arbitrarily large. Thus at some point, the head is deep between two of the counter delimiters. Now, we see a good approximation to a fresh-call configuration when the next recursive call begins.

Next, we observe that we also have control on the state $s$ stored in the $@_s$-symbol: If $s$ is any \emph{recurrent state}, namely a state that appears infinitely many times in the canonical computation, at fresh-call sites arbitrarily far from both ends of the surrounding zero interval. (Some state like this must exist, by the pigeonhole principle.) Then by applying the above argument to a fresh simulation $...000@_t12000...$, we eventually see arbitrarily good approximations of $...000@_s12000...$ during searches through zero zones (and crucially, these are now independent of $t$). Fix one recurrent state $s_*$ once and for all.

Next, we argue that the same happens in all germ-aperiodic orbits. Since the global arrow $a$ can be flipped only once, it suffices to consider a long time-segment where it does not flip. Furthermore, it suffices to assume $a=0$, since the case $a=1$ is symmetric (we are proving a claim about syndetic visits to a clopen set, and this type of property is preserved when inverting orbits).

Suppose now $x$ is a germ-aperiodic point for $H$ where the global arrow is $a=0$, and the global arrow is not flipped in the near future, so that we simulate $M$. We observe that as long as the area seen by the head of $H$ is terra cognita, we are simply following an $M_0$-simulation on an innermost $@0^m10^n2$ segment. This is because, as discussed above, we can ignore all recursive calls by waiting for them to return, assuming they all return. (If some such call does not return, then it is an infinite simulation of the canonical computation, where we have already established recurrence.)

After ignoring the recursive calls, if $d = +$ in this simulation, then we simulate the counter machine $C$ forward in time, and since the inverse of $C$ is mortal, the computation has three options: it runs forever, it eventually halts, or it runs out of space. If it runs forever, because $C^{-1}$ is mortal, the counters eventually get large, and thus eventually we see canonical computations of arbitrary length started recursively. For these we have already established recurrence. Since we are assuming the global arrow is not flipped, we may assume the computation does not halt. It remains to consider the third case, that we run out of space. In this case, the internal $d$-arrow eventually gets flipped, and we are now in the $d = -$ situation.

So consider this $d = -$ situation. In this case, again since $C$ has mortal inverse, eventually the simulation reaches a state with no predecessor. Since we assume the global arrow is not flipped, the only possibility is that we reach $@_s12$, and we step to the parent layer of (what is expected to be) terra cognita. A crucial observation now is that, if indeed at this point we are inside a $0$-zone of a valid counter machine simulation $@0^m10^n2$, then the counter value is strictly larger than the sum of counter values than in the simulation we started with (because the zero contains the initial one as a subword). Even if we are not in such a valid zone, there can only be zeroes in the area that the machine has visited so far.

If at any point, $M$ encounters it is not on terra cognita, by running into an incorrect delimiter during a search, it halts and the global arrow is flipped. Thus, as long as the global time arrow is not flipped, we only see terra cognita, and as explained above, we must be eventually simulating $C$ with arbitrarily large counters, thus we indeed see fresh-call configurations started with the recurrent state $s_*$ 

Now we take $K_a$ the $H$-orbit closure of any such fresh-call configuration $x_a$ with global arrow $a$ and recurrent state $s_*$ at the $@$-symbol. We observe that $x_a$ is germ-aperiodic. Again, this is immediate from our previous claims: The canonical computation is simulated by $M_0$, up to ignoring all recursive simulations (which all eventually run out of space and thus do not halt), and the canonical computation never halts. Since $C$ has mortal inverse, the counters become arbitrarily large, so the in the moving head model, the head must travel arbitrarily far, and later return back to the origin. This proves germ-aperiodicity of $x_a$.

Next, if the global arrow stays $a$ for a sufficiently long time, then as explained above, we must see arbitrarily good approximations to $x_a$. This follows from a simple compactness argument, which we explain for completeness.

The set of germ-aperiodic points $\mathcal{A}(H)$ is compact, write $Z_{a, n} \subset \mathcal{A}(H)$ for the clopen subset of points where the global arrow stays constant $a$ for $n$ steps backward and forward in time. The set of points where the global arrow stays fixed forever is the compact subset $Z_a = \bigcap_n Z_{a,n} \subset \mathcal{A}(H)$. For each $\epsilon > 0$ there exists $n_0$ such that each $x \in Z_a$ enters the $\epsilon$-ball $U_\epsilon$ around $x_a$ every $n_0$ steps. Thus $Z_a \subset \bigcup_{i = 0}^{n_0} H^{-i}(U_\epsilon)$. For large enough $n$, we then have $Z_{a,n} \subset \bigcup_{i = 0}^{n_0} H^{-i}(U_\epsilon)$ as well.

In particular, the orbit closure $K_a$ of $x_a$ is minimal, because $x_a$ is germ-aperiodic, and as the global array stays at value $a$ throughout the orbit of $x_a$, every point $x'$ in the orbit closure of $x_a$ has to have $x_a$ in its orbit closure.
\end{proof}

\subsection{A one-visit computation with constant global arrow}

\begin{lemma}[One-visit lemma]\label{lem:one-visit}
Suppose $T$ does not halt. Then there is a bi-infinite orbit $(y_j)_{j\in\Z}$, $y_j \in X' \times S \times \Z$ in the moving head model of $M$, such that:
\begin{enumerate}
\item the head remains in $[0,\infty)$, all symbols in $(-\infty,-1]$ are error symbols $\#$, which no transition or inverse transition in this computation attempts to read;
\item at time $0$ (i.e.\ in $y_0$) the head is at coordinate $0$; this is the only visit to coordinate $0$;
\item after adjoining $a=0$ in the state, in the corresponding $H$-orbit in the moving-tape model (the standard action of $H$ on $X' \times S \times \{0,1\}$), configurations accumulate on $K_0$ both in the forward and backward time tail.
\end{enumerate}
In particular $((y_j,0))_j$ is an $H$-orbit with constant global arrow $0$. The positive head positions are necessarily unbounded.
\end{lemma}
\begin{proof}
This is the same construction as used for the one-visit property in \cite{Salo}; we record which machine and which direction bit are involved. Consider a valid parent counter zone
\[
 @_s0^m10^n2
\]
with the head at its left delimiter, and take $m$ large. We may fix the parent tag and phase while letting $m\to\infty$. The incoming sweep has just tested the counters, while the outgoing sweep starts the next inspection or update. From this checkpoint, both forward and backward traversal enter the long first zero block. Within bounded distance they initiate a fresh canonical child call. In the backward
traversal the child's internal traversal direction is reversed, as
prescribed by the reversible search routine; the child still represents
the same canonical counter computation. (This is a statement about $M$
and $M^{-1}$, before the completion bit $a$ is adjoined.)

A child in a finite zero block returns only after exhausting its available
space, undoing its work, and restoring the saved parent state. Since the
canonical computation does not halt, the parent checkpoint cannot be reached again until such a return. As $m$ grows, the return times in both directions tend to infinity. Indeed, any fixed finite part of the canonical child computation fits in the block once $m$ is sufficiently large; reversibility gives the same assertion for the reverse traversal. Thus, we have one visit at coordinate $0$. All work in recursive calls happens properly to the right.

Passing to any limit point, and replacing the left tail with $\#$-symbols, we obtain the configuration $y_0$. Note that the machine indeed never inspects or visits these symbols. In fact, we can take simply 
\[ y_0 = (...\#\#\#. @_s0^{\infty}, t, 0). \]
for any $s$ and any suitable state $t$.
\end{proof}

\section{Encoding and the two homeomorphisms}\label{sec:encoding}
\subsection{Simulated tape segments}

Let $X$ be an infinite mixing shift of finite type. Fix a memory radius $r$ for $X$. We use equal-length coding words with two properties: they are \emph{exchangeable}, meaning that any occurrence of one may be replaced by another in the same exterior context without leaving $X$, and any concatenation of them is allowed in $X$. These properties can be obtained by taking long words whose first and last $r$ symbols agree with the same phases of a fixed periodic point of $X$. Mixing and positive entropy provide arbitrarily many such words of one sufficiently large length, chosen to be a multiple of that periodic point's period.

Choose such a family of words of length $\ell$, partitioned into
\[
\begin{aligned}
 W&=\{w_t:t\in A\},\\
 V&=\{v_{q,a,b,t}:q\in S,\ a,b\in\{0,1\},\ t\in A\},\\
 F&=\{u_0,u_1\}.
\end{aligned}
\]
Thus $|A|+4|S||A|+2$ words suffice, and the sets $W,V,F$ are disjoint. A $W$-word stores one tape symbol of $M$. A $V$-word stores that tape symbol, the state $q$ of $M$, the global arrow $a$ of $H$, and a new \emph{head bit} $b$. An $F$-word stores a \emph{tape bit} $c$.

Call a point \emph{active} if its block $[0,\ell-1]$ belongs to $V$. All maps below fix the clopen complement of the active set. At an active point, partition the tape into blocks $[j\ell,(j+1)\ell-1]$, $j\in\mathbb Z$. The maximal interval $I\subset\mathbb Z$ containing $0$ whose other blocks belong to $W$ is the simulated tape segment. The block at $0$ is its current head position. The blocks outside this maximal segment are considered to begin an infinite tail of $\#$-symbols, from the point of view of the simulated machine $H$.

We now simulate the machine $H$ on $I$. For example, when the simulated head moves one cell to the right, we replace the old head block by the appropriate $W$-word, replace the next $W$-word by the new $V$-word, and shift the actual configuration left by $\ell$. The new head block is again at $[0,\ell-1]$. Left moves are analogous. Only a bounded number of blocks is inspected or rewritten in each step. Exchangeability and legality of concatenations of the words ensure that these rewrites stay in $X$. Furthermore, since we imagine $\#$-tails at the ends of $I$, which are never modified by $H$ we indeed obtain a homeomorphism $t \in G_X^u$ which preserves $b$ and simulates $H$ on infinite tapes.

Any finite segment has finitely many tape contents, head states, and head positions relative to its endpoints. Its completed dynamics is therefore periodic.

The role of the tape bit is as follows. If the block immediately preceding the left end of $I$ is $u_c\in F$, then when the head is at the leftmost cell of $I$, we flip the bit $c$. More precisely, let $E$ be the clopen set of active configurations where $[-\ell,-1]$ equals $u_0$ or $u_1$. On $E$ define two commuting local involutions: $B_c$ exchanges the tape bit in $[-\ell,-1]$ ($u_0 \leftrightarrow u_1$), and $B_b$ flips the value of the current head bit $b$. Both are the identity outside $E$.

Their common support condition is unchanged by either involution. Set
\[
                       f=B_c\circ t,\qquad
                       g=B_b\circ B_c\circ t.
\]
Note that $B_c, B_b$ act nontrivially exactly once over the orbit described in Lemma~\ref{lem:one-visit}. All three maps $t,B_c,B_b$ lie in $G_X^u$, hence so do $f,g$.

In plain language, both maps simulate $H$ on encoded tapes, both maps flip a tape bit $c$ whenever the head visits the cell immediately to the right of it, and $g$ additionally flips the head bit $b$ in this situation. In particular $f$ preserves the head bit $b$ on every active orbit.

\subsection{The halting case: an explicit involutory conjugacy}

\begin{lemma}\label{lem:halting-conjugacy}
If $T$ halts, then $t$, $f$, and $g$ have finite order, and an involution
$k\in G_X^u$ satisfies $kfk=g$.
\end{lemma}

\begin{proof}
If $T$ halts, choose $P>0$ with $H^P=\id$ by Lemma~\ref{lem:KO-two-arrows}. The net head displacement of $H^P$ is zero on every configuration: otherwise, varying a remote tape symbol while keeping the finite computation unchanged would contradict $H^P=\id$ on the full shift. Complete any simulated segment by the virtual $\#$-tails used in its definition. Its $t$-orbit follows the corresponding $H$-orbit, which restores the tape, state, and head position after $P$ steps. Hence $t^P=\id$, uniformly for finite, one-sided infinite, and bi-infinite segments. Let
\[
 B=\bigcup_{j=0}^{P-1}t^{-j}E.
\]
where $E$ is the clopen set from the previous section. This is a clopen $t$-invariant set. It is also invariant under $B_b, B_c$, thus it is invariant under $f$ and $g$. At every point of $B$, the left boundary carrying $c$ is reachable in fewer than $P$ steps. By continuity of one-head machines, that boundary is in a fixed finite window around the head. Consequently its current tape bit $c(x)$ is a locally readable function on $B$. Define
\[
 q(x)=\begin{cases}c(x),&x\in B,\\0,&x\notin B,\end{cases}
\]
and let $k$ toggle $b$ by $q(x)$, fixing all other data and all inactive
points. It is a length-preserving local involution: $q$ is independent
of $b$, and changing the head word does not change $B$ or $c$.

Write $e(x)=1_E(t x)$ for the event performed in one step. On $B$ the
only change to the tape bit is the prescribed flip, so
\[
 q(fx)=q(x)+e(x)\pmod2.
\]
Off $B$ both sides are zero. The head bit after $kfk$ is therefore
$b+q(x)+q(fx)=b+e(x)$, and all remaining coordinates evolve exactly as
under $f$. This is precisely $g$. Finally $f^{2P}=\id$: the underlying
$t$-computation repeats after $P$ steps, and its tape-bit parity cancels
after two repetitions. The conjugacy gives the same conclusion for $g$.
\end{proof}

\begin{remark}
The slick proof above slightly encrypts the idea (from \cite{Salo}). The idea is that to separate the dynamics of $f, g$ in the case where the original Turing machine $T$ does not halt, we add the head bit, which only $g$ will flip at each visit. In the case where $T$ does halt, the conjugacy map can see any tape bit that may have been affected during the entire orbit. We can in this case use a cross-section for the dynamics of $t$-orbit, and pick a specific configuration in the orbit of $x$, which we decide to map trivially. Then the cocyclic correction to the head bit is uniquely determined. Unfortunately, this gives a well-defined map only if we ensure that the $f$- and $g$-orbits of $x$ are of the same length, so we add a tape bit, flipped by both maps, that ensures this. It turns out that we can then give an explicit formula for the cocycle as above, and we see it even happens to be an involution, and in fact independent of the choice of cross-section.
\end{remark}

\subsection{The four germ-aperiodic minimal subsystems}

Now suppose that $T$ does not halt. In an infinite coded tape (without a left or right boundary), neither head nor tape bits are ever flipped under actions of $f, g$. The encoding of $K_a$ with head bit $b$ therefore gives a minimal subsystem $X_{a,b}$ on which $f$ and $g$ agree. The four sets $X_{a,b}$ are pairwise disjoint: $a$ and $b$ are locally visible at the head. Let $\xi_{a,b}\in X_{a,b}$ be the encoding of $x_a$ with head bit $b$.

\begin{lemma}\label{lem:four-minimal}
The four sets $X_{a,b}$, for $(a,b)\in\{0,1\}^2$, are exactly the minimal
subsystems of both $f$ and $g$, when restricted to their respective germ-aperiodic sets.
\end{lemma}

\begin{proof}
The block encoding identifies the dynamics of either map on $X_{a,b}$
with the dynamics of $H$ on $K_a$, so $X_{a,b}$ is minimal. Points without
an active head block, and points with a finite simulated segment, are
locally periodic (the tape and head bits can at most enlarge the finite
period). They do not belong to the germ-aperiodic set.

For a germ-aperiodic active orbit, consider one of its infinite constant-arrow
tails. Missing tape cells can be treated as hard errors; after its possible
boundary reversal, the tail faithfully follows a partial computation of $M$
in the selected direction. Lemma~\ref{lem:KO-two-arrows} implies that simulated configurations approach $x_a$. Such approximations occur arbitrarily far from the segment boundary, since the approximating windows contain only the canonical tape symbols, not a hard error or a boundary block. The tape bit is therefore outside these windows.

At each approximation the head bit has one of two values. Passing to a subsequence as the window radius tends to infinity fixes a value $b$ and gives $\xi_{a,b}$ in the orbit closure. Any minimal subsystem of the germ-aperiodic set consequently contains some $\xi_{a,b}$, hence its orbit closure $X_{a,b}$. Being minimal, it must be equal to this $X_{a,b}$.
\end{proof}

\section{The dynamical obstruction}
\label{sec:obstruction}

For a homeomorphism $h$ of Cantor space $C$, let $\mathcal{X}$ be its set of distinct minimal subsystems for the restriction of $h$ to its germ-aperiodic set. For $L_1, L_2 \in \mathcal{X}$ write $L_1 \leadsto L_2$ if for every $\epsilon > 0$ there is a point $x \in C$ such that $d(x, L_1) < \epsilon$ and $d(h^n(x), L_2) < \epsilon$ for some \emph{transition time} $n > 0$. (We do not require germ-aperiodicity of $x$.)
Write $\mathcal{G}_h$ for the graph with nodes $\mathcal{X}$ and edges $(L_1, L_2) \in \mathcal{X}^2$ with $L_1 \leadsto L_2$. Of course, the structure of this graph up to isomorphism is an invariant of topological conjugacy.

Define $\mathsf P_3(h)$ as the property that $\mathcal{G}_h$ contains a directed injective path with three nodes. Thus, there are three fixed, distinct minimal subsystems $L_1,L_2,L_3\subset\mathcal A(h)$ such that, for every $\epsilon>0$ and $i = 1, 2$, one can find a point $x \in C$ and $0<n$ such that $d(x, L_i) < \epsilon \wedge d(h^n(x), L_{i+1}) < \epsilon$.

\begin{lemma}
\label{lem:three-minimal}
If $T$ does not halt, then $\mathsf P_3(g)$ holds and $\mathsf P_3(f)$ fails. Consequently $f$ and $g$ are not topologically conjugate.
\end{lemma}

\begin{proof}
For $g$, fix $L_1=X_{0,0}$, $L_2=X_{0,1}$, and $L_3=X_{1,1}$. We show the stronger result that one can find a single periodic point close to $L_1$, which after a while gets close to $L_2$, and then eventually gets close to $L_3$ (and eventually, of course, repeats). First, use the $M$-computation of Lemma~\ref{lem:one-visit}, encoded on a right-infinite segment with a tape bit at its left boundary (so the actual left $\#$-tail of $y_0$ corresponds to the imagined left $\#$-tail where the encoding blocks in $I$ end), to get a configuration $z$ corresponding to $y_0$, with global arrow $a = 0$.

During the simulated $M$-orbit on $y_0$, the head visits the cell $0$ exactly once in the moving head dynamics. Set the head bit to be $b=0$ before entering this cell. Upon entering cell $0$, it changes to $b=1$ in the $g$-dynamics. Choose sample times on either side of this event, at which canonical recurrence makes the configurations close to $X_{0,0}$ and $X_{0,1}$, specifically $n_0 < 0$ such that $d(g^{n_0}(z), X_{0,0}) < \epsilon$ and $n_1 > 0$ such that $d(g^{n_1}(z), X_{0,1}) < \epsilon$, respectively. Both sample windows lie far from the segment boundary. We may make them as large as needed for the prescribed $\epsilon$.

The subsequent head positions are unbounded to the right; otherwise only finitely many cells would be inspected and the computation would be periodic in the moving-head model. Choose a cell that will subsequently be inspected, further right than all cells inspected between the first and second sample times and outside both chosen observation windows, and insert a hard error symbol $\#$ there, changing $z$ to $z'$. Note that if we make this change sufficiently far from the origin, we still have $d(g^{n_0}(z'), X_{0,0}) < \epsilon$ and $d(g^{n_1}(z'), X_{0,1}) < \epsilon$.

Now observe that at the first attempt to read the new error symbol of $z'$, $M$ has no forward step, so $H$ changes the global arrow to $a = 1$. From that point the completed computation retraces the same partial $M$-path. Before returning to the left boundary it therefore passes the second sample configuration again at some time $n_2$, but with the global arrow changed to $1$ and with the head bit still $1$. Since $\rho K_0=K_1$, this gives the required approximation to $X_{1,1}$.

As mentioned, such individual points $z'$ (constructed for all $\epsilon \rightarrow 0$) already establish $P_3(g)$, concretely one can use the points $g^{n_0}(z')$ (with transition time $n_1 - n_0$) to establish $X_{0,0} \leadsto X_{0,1}$ and the points $g^{n_1}(z')$ (and transition time $n_2 - n_1$) to establish $X_{0,1} \leadsto X_{1,1}$.

For $f$, the head bit $b$ stays constant along every coded orbit. Thus, the only possible edges in $\mathcal{G}_f$ are between $X_{0, b}$ and $X_{1, b}$, . We conclude that there can be no directed path with three vertices in this graph. Hence $\mathsf P_3(f)$ fails. More precisely, choose $\epsilon_0$ smaller than the distance from each $X_{a,b}$ to the clopen set where the head bit has the other value. An orbit of $f$ cannot approach minimal subsystems with both values of $b$ within $\epsilon_0$. A point close enough to any $X_{a,b}$ has a coded head block, and the simulation never changes a coded head block to an uncoded one.
\end{proof}

\begin{lemma}\label{lem:idle-coordinate}
Let $h$ be a homeomorphism of a compact metrizable space $Z$, and let $Y$
be a nonempty compact metrizable space. Then
\[
             \mathsf P_3(h\times\mathrm{id}_Y)
             \quad\Longleftrightarrow\quad \mathsf P_3(h).
\]
where $\mathsf P_3$ is defined as in the previous section.
\end{lemma}

\begin{proof}
First,
\[
 \mathcal A(h\times\mathrm{id}_Y)=\mathcal A(h)\times Y.
\]
Indeed, a power of the product is the identity on a neighborhood of $(z,y)$
if and only if that power of $h$ is the identity on a neighborhood of $z$;
use a product neighborhood for the reverse implication.

Next, every minimal subsystem of the product has the form $L\times\{y\}$, with $L$ a minimal subsystem of $h$. The $Y$-coordinate is constant on every orbit, so the graph $\mathcal{G}$ has no edges that involve multiple $Y$-fibers. Next if $X_1 \leadsto X_2$ in $\mathcal{G}_h$, then $X_1 \times \{y\} \leadsto X_2 \times \{y\}$ in $\mathcal{G}_{h \times \id}$ for any $y \in Y$, since the $Y$-coordinate is never changed. Conversely, if $X_1 \times \{y\} \leadsto X_2 \times \{y'\}$ in $\mathcal{G}_{h \times \id}$ then of course $y = y'$, and $X_1 \leadsto X_2$ in $\mathcal{G}_h$.

We conclude that the isomorphism classes of connected components in the graphs are the same, in particular $P_3(h)$ holds iff $P_3(h \times \id)$ holds.
\end{proof}

Applying Lemma~\ref{lem:idle-coordinate} to $f$ and $g$ proves that the
nonhalting pair remains nonconjugate after taking products with an identity
map. This covers the idle coordinates in the usual embedding of $2V$ into
$nV$, as well as a nonempty finite idle state set.

\begin{proof}[Proof of Theorem~\ref{thm:MainIntro}]
All local rules, state sets, and coding words above are computable from
$T$. The halting case is Lemma~\ref{lem:halting-conjugacy}. In the
nonhalting case Lemma~\ref{lem:three-minimal} distinguishes $f$ and $g$
by $\mathsf P_3$, which is invariant under topological conjugacy.
Lemma~\ref{lem:idle-coordinate} preserves this distinction after adjoining
any nonempty compact metrizable idle coordinate.
\end{proof}

\begin{remark}
\label{rem:Alternatively}
We can alternatively define $L_1 \leadsto L_2$ by requiring that a point exists containing $L_1$ in its $\alpha$-limit set, and $L_2$ in its $\omega$-limit set. Then $L_1 \leadsto L_2 \leadsto L_3$ can be proved for $g$ by separately showing that we can flipping the time arrow and the head bit, which are easier tasks what we did above. However, we find that the existence of a single periodic point visiting a neighborhood of all three a more interesting obstruction.
\end{remark}

\section{Consequences for groups and computability}\label{sec:consequences}

A \emph{listed group} is a countably infinite group $G$ equipped with a surjection $\phi_G:\N\to G$; repetitions are allowed. Computability of a relation $R\subseteq G^2$ refers to its set of names
\[
 \{(m,n)\in\N^2:(\phi_G(m),\phi_G(n))\in R\}.
\]
An \emph{effective homomorphism} $\tau:H\to G$ is a group homomorphism admitting a total computable function $F:\N\to\N$ with $\tau(\phi_H(n))=\phi_G(F(n))$. It is an \emph{effective embedding} if $\tau$ is injective. We do not require decidable equality or computable group operations for an arbitrary listed group.

The groups $G^u_{X,S}$ and $G_{X,S}$ are listed by their finite local rules. Valid rules and equality can be checked effectively for a mixing SFT; alternatively, invertibility can be certified by a second local rule and checking both compositions. Removing repetitions gives a bijective listing with computable multiplication and inversion.

A finitely generated group is listed by words over a fixed finite generating set and its inverses. The local-rule and word listings of $2V=G_{\{0,1\}^{\Z}}$ are effectively equivalent: words can be evaluated as local rules, and a word for a given local rule can be found by enumeration and decidable equality in $2V$.

\begin{corollary}\label{cor:supergroups}
Let $X$ be an infinite mixing SFT and let $G \leq \Homeo(X)$ be any listed group containing $G_X^u$, such that the inclusion $G_X^u\hookrightarrow G$ is effective. Then each relation $\CP_i(G)$, $1\leq i\leq4$, is
undecidable. Moreover $\CP_1(G)$ and $G^2\setminus\CP_4(G)$ are
recursively inseparable.
\end{corollary}

\begin{proof}
Apply Theorem~\ref{thm:MainIntro} and the effective inclusion. A halting input gives a pair in $\CP_1(G)$, whereas a nonhalting input gives a pair outside $\CP_4(G)$. Lemma~\ref{lem:Inclusions} sandwiches each of the four relations between these alternatives. A decision algorithm would therefore decide halting. Likewise any recursive set containing $\CP_1(G)$ and disjoint from $G^2\setminus\CP_4(G)$ would decide halting on the constructed pairs, proving recursive inseparability.
\end{proof}

\begin{corollary}
If $G\leq\Homeo(\{0,1\}^{\Z})$ is finitely generated, listed by words over generators, and contains the standard copy of $2V$, then each relation $\CP_i(G)$, $1\leq i\leq4$, is undecidable. Moreover $\CP_1(G)$ and $G^2\setminus\CP_4(G)$ are recursively inseparable.
\end{corollary}

\begin{proof}
The group $2V$ has $G^u_{\{0,1\}^\Z}$ as a subgroup with an effective inclusion, and generators of $2V$ can be presented by finite words in the generators of $G$, so we obtain an effective embedding of $G^u_{\{0,1\}^\Z}$ in $G$, and the previous corollary applies.
\end{proof}

\begin{proposition}\label{prop:standard-actions}
In their standard actions, $nV$ for $n\geq2$ and $\RTM(k,s)$ for
$k\geq2$, $s\geq1$, have undecidable relations $\CP_i$, $1\leq i\leq4$.
The same holds for listed supergroups of homeomorphisms of the same spaces, containing the indicated standard actions by effective inclusions. In each case $\CP_1$ is recursively inseparable from the complement of $\CP_4$.
\end{proposition}

\begin{proof}
For $2V$, apply the theorem to the binary full shift and use $G_X^u\leq G_X=2V$. For $n>2$, the standard inclusion acts trivially on the remaining Cantor coordinates. The halting conjugator extends by the identity. The nonhalting conclusion persists by taking those coordinates, together with any further idle space, as the space $Y$ in Theorem~\ref{thm:MainIntro}.

For $\RTM(k,s)$, apply the theorem on the $k$-symbol full shift and leave the $s$ external states idle. This gives precisely a product with a finite identity map, which is also covered by the theorem.

It remains to justify automorphism-conjugacy in the larger spaces. The standard $nV$ action is locally dense: inside a clopen rectangle one can exchange smaller prefix rectangles and move a point into any prescribed subrectangle. For $\RTM(k,s)$, finite local word swaps supported in a cylinder and a fixed state make the orbit dense in that state cylinder, just as in Lemma~\ref{lem:Inclusions}. These actions remain locally dense after passage to supergroups. Rubin's theorem therefore gives $\CP_2\subseteq\CP_3$ in every asserted case. The remaining conclusions follow from the same halting reduction.
\end{proof}

\begin{lemma}\label{lem:aut-re}
If a group $G$ is finitely presented and has decidable word problem, then both ordinary conjugacy and automorphism-conjugacy are recursively enumerable relations on words.
\end{lemma}

\begin{proof}
For ordinary conjugacy enumerate candidate conjugating words and test the required equality. For automorphism-conjugacy, fix a finite presentation $\langle x_1,\ldots,x_d\mid R\rangle$. Enumerate pairs of $d$-tuples of words $(u_1,\ldots,u_d)$ and $(v_1,\ldots,v_d)$. Check that each tuple satisfies every relator in the finite set $R$, so that the substitutions define endomorphisms $\alpha$ and $\beta$. Then check $\alpha\beta(x_i)=x_i=\beta\alpha(x_i)$ for all generators. These finite checks are decidable by the word problem, and exactly the pairs of mutually inverse automorphisms pass. Every automorphism occurs with its inverse. Testing $\alpha(f)=g$ in this enumeration semidecides automorphism-conjugacy.
\end{proof}

\begin{theorem}\label{thm:sigma1}
For every $n\geq2$, $\CP_1(nV)$ and $\CP_2(nV)$ are $\Sigma^0_1$-complete under computable many-one reductions.
\end{theorem}
\begin{proof}
The groups $nV$ are finitely presented \cite{BrinPres,HennigMatucci}, and their finite prefix-table descriptions give decidable word problem. Lemma~\ref{lem:aut-re} supplies the $\Sigma^0_1$ upper bounds. The effective halting reduction in Proposition~\ref{prop:standard-actions} supplies hardness for both relations. Finite local or prefix descriptions can be converted effectively to words in a fixed generating set by enumerating words and comparing the resulting tables.
\end{proof}

\begin{remark}
\label{rem:NotRequire}
The preceding argument does not require any effective description of the full automorphism group of $nV$. However, a transducer description of these automorphism groups is available for every $nV$ \cite{Elliott}.
\end{remark}

\begin{question}
What is the exact computability-theoretic complexity of $\CP_3(nV)$ for $n\geq2$?
\end{question}

The conjugating homeomorphism of a Cantor space can be coded by compatible finite clopen partitions, giving the usual analytic upper bound $\Sigma^1_1$. The argument here proves undecidability but does not give a sharper upper bound.

For the moving-head model, we have not investigated the difficulty of the topological conjugacy problem. Of course, we suspect it is at least undecidable.

\begin{question}
How difficult is deciding topological conjugacy of elements of $G^u_{X, S}$ in the moving-head model?
\end{question}

\section*{AI acknowledgements}

This paper was written with assistance from ChatGPT Astra 6 and Claude Sonnet 5.5. Specifically, the following mathematical observations are due to Astra: the observation that the second time arrow in the Kari-Ollinger machine prevents essential minimality, Remark~\ref{rem:NotRequire}, the cross-section free proof of Lemma~\ref{lem:halting-conjugacy} (which was only implicit in \cite{Salo}) and the proof of Lemma~\ref{lem:aut-re} based on general properties (instead of \cite{Elliott}). Remark~\ref{rem:Alternatively} is essentially due to Claude.


\begin{thebibliography}{99}
\bibitem{BaKaSa16}
Sebasti\'an Barbieri, Jarkko Kari, and Ville Salo.
\emph{The group of reversible Turing machines}.
In \emph{Cellular Automata and Discrete Complex Systems}, LNCS 9664,
pages 49--62, 2016.
\url{https://doi.org/10.1007/978-3-319-39300-1_5}.

\bibitem{BaKaSa23}
Sebasti\'an Barbieri, Jarkko Kari, and Ville Salo.
\emph{The group of reversible Turing machines: subgroups, generators, and
computability}.
Forum of Mathematics, Sigma 13:e176, 2025.
\url{https://doi.org/10.1017/fms.2025.10118}.

\bibitem{Brin}
Matthew G. Brin.
\emph{Higher dimensional Thompson groups}.
Geometriae Dedicata 108:163--192, 2004.
\url{https://doi.org/10.1007/s10711-004-8122-9}.

\bibitem{BrinPres}
Matthew G. Brin.
\emph{Presentations of higher dimensional Thompson groups}.
Journal of Algebra 284(2):520--558, 2005.
\url{https://doi.org/10.1016/j.jalgebra.2004.10.028}.

\bibitem{HennigMatucci}
Johanna Hennig and Francesco Matucci.
\emph{Presentations for the higher-dimensional Thompson groups $nV$}.
Pacific Journal of Mathematics 257(1):53--74, 2012.
\url{https://doi.org/10.2140/pjm.2012.257.53}.

\bibitem{SMART}
Julien Cassaigne, Nicolas Ollinger, and Rodrigo Torres-Avil\'es.
\emph{A small minimal aperiodic reversible Turing machine}.
Journal of Computer and System Sciences 84:288--301, 2017.
\url{https://doi.org/10.1016/j.jcss.2016.10.004}.

\bibitem{Elliott}
Luke Elliott.
\emph{A description of $\operatorname{Aut}(dV_n)$ and $\operatorname{Out}(dV_n)$ using transducers}.
Groups, Geometry, and Dynamics 17(1):293--312, 2023.
\url{https://doi.org/10.4171/GGD/697}.

\bibitem{GajardoOllingerTorres}
Anah\'i Gajardo, Nicolas Ollinger, and Rodrigo Torres-Avil\'es.
\emph{The transitivity problem of Turing machines}.
In \emph{Mathematical Foundations of Computer Science 2015}, LNCS 9234,
pages 231--242, 2015.
\url{https://doi.org/10.1007/978-3-662-48057-1_18}.

\bibitem{KurkaTuringMachineDynamics}
Petr K\r{u}rka.
\emph{On topological dynamics of Turing machines}.
Theoretical Computer Science 174(1--2):203--216, 1997.
\url{https://doi.org/10.1016/S0304-3975(96)00025-4}.

\bibitem{Morita}
Kenichi Morita.
\emph{Universality of a reversible two-counter machine}.
Theoretical Computer Science 168(2):303--320, 1996.
\url{https://doi.org/10.1016/S0304-3975(96)00081-3}.

\bibitem{TorresMixing}
Rodrigo Torres-Avil\'es.
\emph{Topological mixing notions on Turing machine dynamical systems}.
Information and Computation 285, Part B:104915, 2022.
\url{https://doi.org/10.1016/j.ic.2022.104915}.

\bibitem{KaOl08}
Jarkko Kari and Nicolas Ollinger.
\emph{Periodicity and immortality in reversible computing}.
In \emph{Mathematical Foundations of Computer Science 2008}, LNCS 5162,
pages 419--430, 2008.
\url{https://doi.org/10.1007/978-3-540-85238-4_34}.
The extended version, including Appendix A, is
\url{https://hal.science/hal-00270815}.

\bibitem{Rubin}
Matatyahu Rubin.
\emph{On the reconstruction of topological spaces from their groups of
homeomorphisms}. Transactions of the American Mathematical Society
312(2):487--538, 1989.
\url{https://doi.org/10.1090/S0002-9947-1989-0988881-4}.

\bibitem{Salo}
Ville Salo.
\emph{Conjugacy of reversible cellular automata and one-head machines}.
\url{https://arxiv.org/abs/2011.07827}.
\end{thebibliography}
\end{document}